\documentclass[11pt]{article}

\usepackage{epsfig}
\usepackage{amssymb}
\usepackage[]{times}
\newcommand{\be}{\begin{equation}}
\newcommand{\ee}{\end{equation}}
\newcommand{\bea}{\begin{eqnarray}}
\newcommand{\eea}{\end{eqnarray}}
\newtheorem{theorem}{Theorem}

\newtheorem{corollary}{Corollary}

\newtheorem{conjecture}{Conjecture}
\newtheorem{example}{Example}
\newenvironment{proof}[1][Proof]{\begin{trivlist}
\item[\hskip\labelsep {\bfseries #1}]}{\end{trivlist}}
\def\1#1{^{(#1)}}
\def\la{\langle}
\def\ra{\rangle}
\begin{document}
\title{Genera of numerical semigroups and\\
polynomial identities for degrees of syzygies}
\author{Leonid G. Fel\\ \\
Department of Civil Engineering, Technion, Haifa 32000, Israel\\
{\em e-mail: lfel@technion.ac.il}} 
\date{}
\maketitle
\def\be{\begin{equation}}
\def\ee{\end{equation}}
\def\bea{\begin{eqnarray}}
\def\eea{\end{eqnarray}}
\def\p{\prime}
\vspace{-.75cm}
\begin{abstract}
We derive polynomial identities of arbitrary degree $n$ for syzygies degrees of 
numerical semigroups $S_m\!=\!\la d_1,\ldots,d_m\ra$ and show that for $n\ge m$ 
they contain higher genera $G_r\!=\!\sum_{s\in{\mathbb Z}_>\!\!\setminus S_m}s^
r$ of $S_m$. We find a number $g_m\!=\!B_m-m+1$ of algebraically independent 
genera $G_r$ and equations, related any of $g_m+1$ genera, where $B_m\!=\!
\sum_{k=1}^{m-1}\beta_k$ and $\beta_k$ denote the total and  partial Betti 
numbers of non-symmetric semigroups. The number $g_m$ is strongly dependent on 
symmetry of $S_m$ and decreases for symmetric semigroups and complete 
intersections.
\end{abstract}
\section{Introduction}\label{l1} 
Two sets of polynomial and quasi-polynomial identities for degrees of syzygies in
numerical semigroups $S_m\!=\!\la d_1,\ldots,d_m\ra$ were derived recently 
\cite{fl17} when studying the rational representation (Rep) of the Hilbert 
series of $S_m$ and the quasi-polynomial Rep of the restricted partition 
function. A part of polynomial identities of degrees $1\le n\le m-2$ were 
coincided with Herzog-K\"uhl's equations \cite{Her84} for the Betti numbers of
graded Cohen-Macaulay modules of codimension $m\!-\!1$, but a new additional 
polynomial identity of degree $n\!=\!m-1$ turned out to be an important tool in 
a study of symmetric (not complete intersection) semigroups with small embedding
dimension ({\em edim}) $m=4,5,6$ \cite{fl15,f18,fa18}. 

A further application \cite{fl18} of polynomial identities for higher degrees of
syzygies, $n\ge m$, involves power sums $G_{n-m}\!=\!\sum_{s\in\Delta_m}s^{n-m}
$, which called {\em genera} of numerical semigroups, where $\Delta_m\!=\!
{\mathbb Z}_>\!\!\setminus S_m$ and $G_0=\#\Delta_m$ denote a set of semigroup 
gaps and its cardinality ({\em genus}), respectively. A set $\Delta_m$ is 
uniquely defined by semigroup generators $d_j$. Albeit there are $\mu-1$ 
explicitly known gaps $1\le s\le\mu-1$, where $\mu=\min\{d_1,\ldots,d_m\}$ 
denotes a semigroup multiplicity, a most of gaps $s>\mu$ (including the largest 
gap $F_m$ which called the Frobenius number) cannot be determined explicitly.

By a fundamental theorem of symmetric polynomials \cite{mc95}, there exists a 
finite number $g_m$ of algebraically independent genera $G_r$. On the other 
hand, an involvement of $G_r$ into polynomial identities for syzygies degrees 
may decrease this number. In the present paper, we study this question for 
arbitrary semigroup $S_m$ and find how does the number $g_m$ depend on special 
characteristics of a semigroup (e.g., {\em edim}) and its properties 
(non-symmetric, symmetric, complete intersection). For this purpose, we derive 
polynomial identities of higher degrees $n\ge m$ and find algebraic equalities 
related a finite number $g_m+1$ of genera.

The paper is organized in six sections. In section \ref{l2} we obtain polynomial
identities (\ref{y9}) of higher degrees $n\ge m$ following an approach, 
suggested in \cite{Her84}. The rest of this section is completely technical: we 
determine necessary expressions for all entries appeared in formula (\ref{y9}). 
In section \ref{l3} we derive linear equations (\ref{y19}) for alternating power
sums ${\mathbb C}_k(S_m)$ and put forward a conjecture on the linear Rep of 
coefficients $K_p$ in (\ref{y22}) by genera $G_r$ of a semigroup and 
special polynomials $T_r$ defined in (\ref{y25}). In section \ref{l4} we prove 
the existence of a polynomial equation ${\mathcal R}_G(G_0,\ldots,G_{g_m})\!=\!
0$ for arbitrary non-symmetric semigroups, where $g_m\!=\!B_m-m+1$, and $B_m\!=
\!\sum_{k=1}^{m-1}\beta_k$ and $\beta_k$ denote the total and partial Betti 
numbers of $S_m$, and find such equation for $S_3$. In section \ref{l5} 
we discuss supplementary relations for 
$G_k$ in symmetric semigroups and complete intersection (CI), and give formulas 
for $g_m$ in both cases, e.g., in the latter case it looks much simple, $g_m=m
-2$. In section \ref{l6} we give concluding remarks about coefficients $K_p$.
\section{Polynomial identities for numerical semigroups}\label{l2}
Recall the basic facts on numerical semigroups and polynomial identities
following \cite{fl17}. Let a numerical semigroup $S_m$ be minimally generated by
a set of natural numbers $\{d_1,\ldots,d_m\}$, where $\mu\ge m$ and $\gcd(d_1,
\ldots,d_m)=1$. Its generating function $H\left(S_m;z\right)$,
\bea
H\left(S_m;z\right)=\sum_{s\;\in\;S_m}z^s,\qquad z<1,\qquad 0\in S_m,\label{y1}
\eea
is referred to as {\em the Hilbert series} of $S_m$ and has a rational Rep,
\bea
H\left(S_m;z\right)=\frac{Q\left(S_m;z\right)}{\prod_{i=1}^m\left(1-z^{d_i}
\right)},\qquad C_{k,j}\in{\mathbb Z}_>,\;\;1\le k\le m-1,\;\;1\le j\le \beta_k,
\label{y2}
\eea
\vspace{-.5cm}
\bea
Q\left(S_m;z\right)=1-\sum_{j=1}^{\beta_1}z^{C_{1,j}}+\sum_{j=1}^{\beta_2}      
z^{C_{2,j}}-\cdots\pm\sum_{j=1}^{\beta_{m-1}}z^{C_{m-1,j}},\qquad\sum_{k=0}^
{m-1}(-1)^k\beta_k=0,\quad\beta_0=1,\label{y3}
\eea
where $C_{k,j}$ and $\beta_k$ stand for degrees of syzygies and partial Betti's 
numbers, respectively. The largest degree $Q_m$ of the $(m-1)$-th syzygy is 
related to the Frobenius number $F_m$ of $S_m$,
\bea
Q_m\!=\!F_m+\sigma_1,\quad Q_m\!=\!\max{\sf PF}(S_m),\;\;{\sf PF}(S_m)\!=\!
\left\{C_{m-1,1},\ldots,C_{m-1,\beta_{m-1}}\right\},\quad\sigma_1\!=\!
\sum_{j=1}^md_j,\label{y4}
\eea
where ${\sf PF}(S_m)$ is called a set of pseudo-Frobenius numbers. Denote by 
${\mathbb C}_k(S_m)$ the alternating power sum of syzygies degrees,
\bea
{\mathbb C}_k(S_m)=\sum_{j=1}^{\beta_1}C_{1,j}^k-\sum_{j=1}^{\beta_2}C_{2,j}^k  
+\ldots-(-1)^{m-1}\sum_{j=1}^{\beta_{m-1}}C_{m-1,j}^k,\label{y5}
\eea
and write the polynomial identities (Theorem 1 in \cite{fl17}) for a semigroup
$S_m$,
\bea
{\mathbb C}_0(S_m)=1,\qquad{\mathbb C}_r(S_m)=0,\quad 1\le r\le m-2,\qquad
{\mathbb C}_{m-1}(S_m)=(-1)^m(m-1)!\pi_m,\label{y6}
\eea
where $\pi_m=\prod_{i=1}^md_i$.
\subsection{Polynomial identities of arbitrary degree}\label{l21}
Start with relation for the Hilbert series $H\left(S_m;z\right)$ and a
generating function $\Phi\left(S_m;z\right)$ for the semigroup gaps $s\in\Delta
_m$,
\bea
\Phi\left(S_m;z\right)+H\left(S_m;z\right)=\frac1{1-z},\qquad\Phi\left(S_m;z
\right)=\sum_{s\in\Delta_m}z^s,\label{y7}
\eea
and present the numerator $Q\left(S_m;z\right)$ in (\ref{y2}) as follows,
\bea
Q\left(S_m;z\right)=(1-z)^{m-1}\Pi\left(S_m;z\right),\qquad\Pi\left(S_m;z\right)
=\Psi\left(S_m;z\right)\left[1-(1-z)\Phi\left(S_m;z\right)\right],\label{y8}
\eea
where $\Psi\left(S_m;z\right)$ is a product of cyclotomic polynomials $\Psi_j
\left(S_m;z\right)$,
\bea
\Psi\left(S_m;z\right)=\prod_{j=1}^m\Psi_j(z),\qquad\Psi_j(z)=\sum_{k=0}^{d_j-1}
z^k,\quad\Psi_j(1)=d_j,\quad\Psi\left(S_m;1\right)=\pi_m.\nonumber
\eea
Differentiating $r$ times the first equality in (\ref{y8}), we obtain an 
infinite set of algebraic equations related syzygies degrees $C_{k,j}$ of a 
semigroup $S_m$ with its generators $d_j$ and gaps $s\in\Delta_m$,
\bea
Q^{(r)}_z(z)=\sum_{k=0}^r(-1)^k\frac{(m-1)!}{(m-k-1)!}{r\choose k}(1-z)^{m-k-1}
\Pi^{(r-k)}_z(z),\qquad r\ge 1,\label{y9}
\eea
where
\bea
Q^{(r)}_z(z)=\frac{d^rQ\left(S_m;z\right)}{dz^r},\qquad\Pi^{(r)}_z(z)=\frac{d^r
\Pi\left(S_m;z\right)}{dz^r}.\nonumber
\eea

Calculate separately derivatives $Q^{(r)}_z(z)$ and $\Pi^{(r-k)}_z(z)$. 
According to expression (\ref{y3}), we obtain, 
\bea
Q^{(r)}_z(z)\!=\!-\sum_{j=1}^{\beta_1}(C_{1,j})_rz^{C_{1,j}-r}\!+\sum_{j=1}^
{\beta_2}(C_{2,j})_rz^{C_{2,j}-r}\!-\ldots+(-1)^{m-1}\sum_{j=1}^{\beta_{m-1}}
(C_{m-1,j})_rz^{C_{m-1,j}-r},\;\label{y10}
\eea
where
\bea
(C_{i,j})_r=C_{i,j}(C_{i,j}-1)\times\ldots\times(C_{i,j}-r+1),\quad\mbox{if}\;\;
r\le C_{i,j}\quad\mbox{and}\quad(C_{i,j})_r=0,\quad\mbox{if}\;\;r>C_{i,j},
\nonumber
\eea
and $(x)_r=x(x-1)\times\ldots\times (x-r+1)$ denotes the falling factorial. 

Making use of alternating sums ${\mathbb C}_k(S_m)$ in (\ref{y5}), present the
polynomial expansion (\ref{y10}) as follows,
\bea
Q^{(r)}_z(1)=-\sum_{k=0}^r{\mathfrak S}^r_k{\mathbb C}_k(S_m),\qquad{\mathfrak 
S}^n_k=(-1)^{n-k}\left[n\atop k\right],\qquad{\mathfrak S}^n_n=1,\label{y11}
\eea
where ${\mathfrak S}^n_k$ denote Stirling's numbers of the 1st kind and symbols 
$\left[n\atop k\right]$ satisfy the recurrence relation,
\bea
\left[n+1\atop k\right]=n\left[n\atop k\right]+\left[n\atop k-1\right],\quad 1
\le k\le n,\qquad\left[n\atop n\right]=1.\label{y12}
\eea
In Appendix \ref{appendix1} we present the first expressions for $\left[n\atop 
n-k\right]$ up to $k=9$.

A straightforward calculation of the derivative $\Pi^{(r)}_z(z)$ gives
\bea
&&\Pi^{(r)}_z(z)=\Psi^{(r)}_z(z)+\sum_{k=1}^rk{r\choose k}\Psi^{(r-k)}_z(z)
\Phi^{(k-1)}_z(z)-(1-z)\sum_{k=0}^r{r\choose k}\Psi^{(r-k)}_z(z)\Phi^{(k)}_z(z),
\qquad\label{y13}
\eea 
where
\bea
\Psi^{(r)}_z(z)\!=\!\sum_{k_1,\ldots, k_m\ge 0}^{r=k_1+\ldots +k_m}\!\frac{r}
{k_1!\cdots k_m!}\prod_{j=1}^m\Psi_{j,z}^{(k_j)}(z),\quad\Psi^{(k)}_{j,z}(z)=
\!\sum_{l\ge k}^{d_j-1}(l)_kz^{l-k},\quad\Phi^{(k)}_z(z)=\!\sum_{s\in\Delta_m
\atop s\ge k}(s)_kz^{s-k},\nonumber
\eea
and
\bea
\Psi^{(r)}_z(z)=\frac{d^r\Psi\left(S_m;z\right)}{dz^r},\qquad
\Psi^{(r)}_{j,z}(z)=\frac{d^r\Psi_j\left(S_m;z\right)}{dz^r},\qquad
\Phi^{(r)}_z(z)=\frac{d^r\Phi\left(S_m;z\right)}{dz^r}.\label{y14}
\eea
Thus, we obtain an expression for the derivative $\Pi^{(r)}_z(z)$ at $z=1$, 
\bea
&&\Pi^{(r)}_{z=1}=\Psi^{(r)}_{z=1}+\sum_{k=1}^rk{r\choose k}\Psi^{(r-k)}_{z=1}
\Phi^{(k-1)}_{z=1}=\Psi^{(r)}_{z=1}+r\sum_{k=0}^{r-1}{r-1\choose k}\Psi^{(r-k-1)
}_{z=1}\Phi^{(k)}_{z=1},\label{y15}\\
&&\Pi^{(r)}_{z=1}=\Pi^{(r)}_z(z=1),\quad\Psi^{(r)}_{z=1}=\Psi^{(r)}_z(z=1),\;\;
\Psi^{(k)}_{j,z=1}=\Psi^{(k)}_{j,z}(z=1),\quad\Phi^{(k)}_{z=1}=\Phi^{(k)}_z(z=1)
.\quad\nonumber
\eea
\subsection{Derivatives $\;\Phi^{(r)}_{z=1}$, $\Psi^{(r)}_{z=1}$ and $\Pi^{(r)}_
{z=1}$}\label{l22}
Derivatives $\Phi^{(r)}_{z=1}$ may be calculated separately in accordance with 
(\ref{y14}),
\bea
&&\Phi^{(r)}_{z=1}=\sum_{k=0}^r{\mathfrak S}^r_kG_k,\qquad G_k=\sum_{s\in\Delta
_m}s^k,\qquad\mbox{e.g.}\label{y16}\\
&&\Phi^{(0)}_{z=1}=G_0,\qquad\Phi^{(1)}_{z=1}=G_1,\qquad\Phi^{(2)}_{z=1}=G_2-
G_1,\qquad\Phi^{(3)}_{z=1}=G_3-3G_2+2G_1,\nonumber\\
&&\Phi^{(4)}_{z=1}=G_4-6G_3+11G_2-6G_1,\qquad\Phi^{(5)}_{z=1}=G_5-10G_4+35G_3-
50G_2+24G_1.\nonumber
\eea
The sums $G_k$ are known as {\em genera} of numerical semigroup $S_m$ and $G_0$ 
denotes a semigroup {\em genus}.

General formulas for derivatives $\Psi^{(r)}_{z=1}$ and $\Pi^{(r)}_{z=1}$ are 
given in (\ref{y13}) and (\ref{y15}) and cannot be simplified essentially for 
arbitrary $r$. Here we present the expressions for $\Psi^{(r)}_{z=1}$ and 
$\Pi^{(r)}_{z=1}$ for small $r\le 4$. All necessary calculations are given in 
Appendix \ref{appendix2}.
\bea
\frac{\Psi^{(0)}_{z=1}}{\pi_m}&\!\!=\!\!&1,\quad\frac{\Psi^{(1)}_{z=1}}{\pi_m}
=\frac{\sigma_1-m}{2},\quad\frac{\Psi^{(2)}_{z=1}}{\pi_m}=\left(\frac{\sigma_1
-m}{2}\right)^2+\frac{\sigma_2-6\sigma_1+5m}{12},\label{y17}\\
\frac{\Psi^{(3)}_{z=1}}{\pi_m}&\!\!=\!\!&\left(\frac{\sigma_1-m}{2}-1\right)
\left[\left(\frac{\sigma_1-m}{2}\right)^2+\frac{\sigma_2-4\sigma_1+3m}{4}
\right],\nonumber\\
\frac{\Psi^{(4)}_{z=1}}{\pi_m}&\!\!=\!\!&\frac1{3}\left(\frac{\sigma_2-6\sigma_1
+5m}{4}\right)^2+\left(\frac{\sigma_1-m}{2}\right)^2\frac{\sigma_2-6\sigma_1+
5m}{2}+\left(\frac{\sigma_1-m}{2}\right)^4-\nonumber\\
&&\frac{\sigma_1-m}{2}(\sigma_2-4\sigma_1+3m)-\frac{\sigma_4-110\sigma_2+
360\sigma_1-251m}{120},\nonumber
\eea
where $\sigma_k=\sum_{j=1}^md_j^k$ are power sums of generators $d_j$. 
Substituting (\ref{y16},\ref{y17}) into formulas (\ref{y15}), we arrive at 
expressions for $\Pi^{(r)}_{z=1}$,
\bea
\frac{\Pi^{(0)}_{z=1}}{\pi_m}&\!\!=\!\!&1,\qquad
\frac{\Pi^{(1)}_{z=1}}{\pi_m}=\frac{\sigma_1-m}{2}+G_0,\label{y18}\\
\frac{\Pi^{(2)}_{z=1}}{\pi_m}&\!\!=\!\!&\left(\frac{\sigma_1-m}{2}\right)^2+
\frac{\sigma_2-6\sigma_1+5m}{12}+(\sigma_1-m)G_0+2G_1,\nonumber\\
\frac{\Pi^{(3)}_{z=1}}{\pi_m}&\!\!=\!\!&\left(\frac{\sigma_1-m}{2}-1\right)
\left[\left(\frac{\sigma_1-m}{2}\right)^2+\frac{\sigma_2-4\sigma_1+3m}{4}\right]
+\nonumber\\
&&\left[3\left(\frac{\sigma_1-m}{2}\right)^2+\frac{\sigma_2-6\sigma_1+5m}{4}
\right]G_0+3(\sigma_1-m)G_1+3(G_2-G_1),\nonumber\\
\frac{\Pi^{(4)}_{z=1}}{\pi_m}&\!\!=\!\!&\frac1{3}\left(\frac{\sigma_2-6\sigma_1
+5m}{4}\right)^2+\left(\frac{\sigma_1-m}{2}\right)^2\frac{\sigma_2-6\sigma_1
+5m}{2}+\left(\frac{\sigma_1-m}{2}\right)^4-\nonumber\\
&&\frac{\sigma_1-m}{2}(\sigma_2-4\sigma_1+3m)-\frac{\sigma_4-110\sigma_2+
360\sigma_1-251m}{120}+\nonumber\\
&&\left(\frac{\sigma_1-m}{2}-1\right)\left[(\sigma_1-m)^2+\sigma_2-4\sigma_1+3m
\right]G_0+\nonumber\\
&&\left[3(\sigma_1-m)^2+\sigma_2-6\sigma_1+5m\right]G_1+6(\sigma_1-m)(G_2-G_1)
+4(G_3-3G_2+2G_1).\nonumber
\eea
\section{Linear equations for alternating power sums ${\mathbb C}_k(S_m)$}
\label{l3}
In the right hand side of expression (\ref{y9}) for $Q^{(r)}_z(1)$, there 
survives a solely one term, namely, when $k=m-1$. Combining the resulting 
expression in (\ref{y9}) with (\ref{y11}), we obtain,
\bea
(-1)^m(m-1)!\;{r\choose m-1}\Pi^{(r-m+1)}_{z=1}=\sum_{k=m-1}^r{\mathfrak S}^r_k
{\mathbb C}_k(S_m),\qquad r\ge m-1.\nonumber
\eea
Represent the last equation in a more convenient way by shifting the variable 
$r$, i.e., $r=m+p$,
\bea
\sum_{k=m-1}^{m+p}{\mathfrak S}^{m+p}_k{\mathbb C}_k(S_m)=(-1)^m\frac{(m+p)!}
{(1+p)!}\;\Pi^{(p+1)}_{z=1},\qquad p\ge -1.\label{y19}
\eea
Thus, we arrive at the matrix equation with $p+2$ variables ${\mathbb C}_k(S_m
)$, where $k=m-1,\ldots,m+p$,
\bea
\left(\begin{array}{ccccc}{\mathfrak S}^{m-1}_{m-1}&0&0&\ldots&0\\
{\mathfrak S}^m_{m-1}&{\mathfrak S}^m_m&0&\ldots&0\\
{\mathfrak S}^{m+1}_{m-1}&{\mathfrak S}^{m+1}_m&{\mathfrak S}^{m+1}_{m+1}&
\ldots&0\\
\ldots&\ldots&\ldots&\ldots&0\\
{\mathfrak S}^{m+p}_{m-1}&{\mathfrak S}^{m+p}_m&{\mathfrak S}^{m+p}_{m+1}&
\ldots&{\mathfrak S}^{m+p}_{m+p}\end{array}\right)\!
\left(\begin{array}{l}{\mathbb C}_{m-1}(S_m)\\{\mathbb C}_m(S_m)\\
{\mathbb C}_{m+1}(S_m)\\\ldots\\
{\mathbb C}_{m+p}(S_m)\end{array}\right)=(-1)^m\left(\begin{array}{r}(m-1)!\;
\Pi^{(0)}_{z=1}\\m!\;\Pi^{(1)}_{z=1}\\\frac{(m+1)!}{2!}\;\Pi^{(2)}_{z=1}\\
\ldots\\
\frac{(m+p)!}{(p+1)!}\;\Pi^{(p+1)}_{z=1}\end{array}\right),\quad\label{y20}\\
\nonumber
\eea

\noindent
where, according to definition of the Stirling numbers (\ref{y12}), we have 
in a diagonal ${\mathfrak S}^r_r=1$, $r\ge 0$.

The general solution of matrix equation (\ref{y20}) may be written as follows,
\bea
{\mathbb C}_{m+p}(S_m)&=&(-1)^m\frac{(m+p)!}{(p+1)!}\Pi^{(p+1)}_{z=1}-
\sum_{j=1}^{p+1}(-1)^j\left[m+p\atop m+p-j\right]{\mathbb C}_{m+p-j}(S_m),\quad
\mbox{e.g.,}\quad\label{y21}\\
{\mathbb C}_{m-1}(S_m)&=&(-1)^m(m-1)!\;\Pi^{(0)}_{z=1},\nonumber\\
{\mathbb C}_m(S_m)&=&(-1)^mm!\;\Pi^{(1)}_{z=1}+\left[m\atop m-1\right]{\mathbb 
C}_{m-1}(S_m),\nonumber\\
{\mathbb C}_{m+1}(S_m)&=&(-1)^m\frac{(m+1)!}{2}\;\Pi^{(2)}_{z=1}+\left[m+1\atop
m\right]{\mathbb C}_m(S_m)-\left[m+1\atop m-1\right]{\mathbb C}_{m-1}(S_m).
\nonumber
\eea
Calculating ${\mathbb C}_{m+p}(S_m)$ in (\ref{y21}) by consecutive substitution
of ${\mathbb C}_{m+q-1}(S_m)$ into ${\mathbb C}_{m+q}(S_m)$, where $q=0,\ldots,
p$, we arrive at the final expression,
\bea
{\mathbb C}_n(S_m)=\frac{(-1)^mn!}{(n-m)\:!}\;\pi_mK_{n-m},\qquad 
{\mathbb C}_{m-1}(S_m)=(-1)^m(m-1)\:!\:\pi_m,\label{y22}
\eea
where a coefficient $K_p$ is a linear combination of genera $G_0,\ldots,G_p$. 
We present here expressions 
\footnote{Formulas for $K_0,K_1,K_2,K_3$ were calculated by consecutive 
substitution of expressions (\ref{y22}) and (\ref{y18}) into (\ref{y21}). The 
other three formulas for $K_4,K_5,K_6$ were found in two steps: 1) by analytic 
derivations (with help of Mathematica software) of expressions for $\Psi^{(r)}
_{z=1}$ and $\Pi^{(r)}_{z=1}$, $r=5,6,7$, which are extremely lengthy to be 
disposed in the paper, 2) by consecutive substitution of the found expressions 
for $\Pi^{(r)}_{z=1}$ and ${\mathbb C}_{m+r-1}(S_m)$ into (\ref{y21}).}
for $K_p$ when $p\le 6$,
\bea
K_0&=&G_0+\delta_1,\hspace{4cm}\delta_p=\frac{\sigma_p-1}{2^p},\label{y23}\\
K_1&=&G_1+\frac{\sigma_1}{2}G_0+\frac{3\delta_1^2+\delta_2}{6},\nonumber\\
K_2&=&G_2+\sigma_1G_1+\frac{3\sigma_1^2+\sigma_2}{12}G_0+\frac{\delta_1\left(
\delta_1^2+\delta_2\right)}{3},\nonumber\\
K_3&=&G_3+\frac{3}{2}\sigma_1G_2+\frac{3\sigma_1^2+\sigma_2}{4}G_1+\frac{
\sigma_1(\sigma_1^2+\sigma_2)}{8}G_0+\frac{15\delta_1^4+30\delta_1^2\delta_2+5
\delta_2^2-2\delta_4}{60}.\nonumber\\
K_4&=&G_4+2\sigma_1G_3+\frac{3\sigma_1^2+\sigma_2}{2}G_2+\frac{\sigma_1(
\sigma_1^2+\sigma_2)}{2}G_1+\nonumber\\
&&\frac{15\sigma_1^4+30\sigma_1^2\sigma_2+5\sigma_2^2-2\sigma_4}{240}G_0+
\delta_1\frac{3\delta_1^4+10\delta_1^2\delta_2+5\delta_2^2-2\delta_4}{15},
\nonumber\\
K_5&=&G_5+\frac{5}{2}\sigma_1G_4+\frac{5}{6}\left(3\sigma_1^2+\sigma_2\right)G_3
+\frac{5}{4}\sigma_1(\sigma_1^2+\sigma_2)G_2+\nonumber\\
&&\frac{15\sigma_1^4+30\sigma_1^2\sigma_2+5\sigma_2^2-2\sigma_4}{48}G_1+
\sigma_1\frac{3\sigma_1^4+10\sigma_1^2\sigma_2+5\sigma_2^2-2\sigma_4}{96}G_0+
\nonumber\\
&&\frac{63\delta_1^6+315\delta_1^4\delta_2+315\delta_1^2\delta_2^2-126\delta_1^2
\delta_4+35\delta_2\delta_4-42\delta_2^3+16\delta_6}{378}
\nonumber\\
K_6&=&G_6+3\sigma_1G_5+\frac{5}{4}\left(3\sigma_1^2+\sigma_2\right)G_4+
\frac{5}{2}\sigma_1(\sigma_1^2+\sigma_2)G_3+\nonumber\\
&&\frac{15\sigma_1^4+30\sigma_1^2\sigma_2+5\sigma_2^2-2\sigma_4}{16}G_2+
\sigma_1\frac{3\sigma_1^4+10\sigma_1^2\sigma_2+5\sigma_2^2-2\sigma_4}{16}G_1+
\nonumber\\
&&\frac{63\sigma_1^6+315\sigma_1^4\sigma_2+315\sigma_1^2\sigma_2^2-126\sigma_1^2
\sigma_4+35\sigma_2\sigma_4-42\sigma_2^3+16\sigma_6}{4032}G_0+
\nonumber\\
&&\delta_1\frac{9\delta_1^6+63\delta_1^4\delta_2+105\delta_1^2\delta_2^2-42
\delta_1^2\delta_4+35\delta_2^3-42\delta_2\delta_4+16\delta_6}{63}.\nonumber
\eea

A straightforward calculation of higher $K_p$ (even with help of Mathematica 
software) encounters with enormous technical difficulties. On the other hand, 
a careful observation of formulas (\ref{y23}) allows to put forward a conjecture
about a general formula for $K_p$ for arbitrary $p$, which is related to a 
special kind of symmetric polynomials $P_n=P_n(x_1,\ldots,x_m)$ of degree $n$ in
$m$ variables, discussed in \cite{fl20},
\bea
P_n=\sum_{j=1}^mx_j^n-\sum_{j>r=1}^m\left(x_j+x_r\right)^n+\sum_{j>r>i=1}^m
\left(x_j+x_r+x_i\right)^n-\ldots-(-1)^m\left(\sum_{j=1}^mx_j\right)^n.
\label{y24}
\eea
In what follows, we make use of a remarkable property of polynomials $P_n$: its 
factorization reads \cite{fl20},
\bea
P_n(x_1,\ldots,x_m)=\frac{(-1)^{m+1}n!}{(n-m)!}\;\chi_mT_{n-m}(x_1,\ldots,x_m),
\qquad\chi_m=\prod_{j=1}^mx_j,\quad T_0=1,\label{y25}
\eea
where $T_r(x_1,\ldots,x_m)$ is a symmetric polynomial of degree $r$ in $m$ 
variables. According to \cite{fl20}, this polynomial satisfies inequality,
\bea
T_r(x_1,\ldots,x_m)>0,\qquad x_1,\ldots,x_m>0,\label{y26}
\eea

Denote $X_k=\sum_{j=1}^mx_j^k$ and, by a fundamental theorem of symmetric 
polynomials \cite{mc95}, use such power sums as a basis for algebraic Rep of 
polynomials $T_r$. In other words, instead of $T_r(x_1,\ldots,x_m)$, we make use
in (\ref{y25}) of polynomials $T_r(X)=T_r(X_1,\ldots,X_r)$, which were derived 
in \cite{fl20} and presented in Appendix \ref{appendix3}. To pose a conjecture, 
define two polynomials $T_r(\sigma)=T_r(\sigma_1,\ldots,\sigma_r)$ and $T_r(
\delta)=T_r(\delta_1,\ldots,\delta_r)$ by replacing $X_k\to\sigma_k$ and $X_k
\to\delta_k$ in $T_r(X_1,\ldots,X_r)$, where $\sigma_k$ and $\delta_k$ are 
defined in (\ref{y17}) and (\ref{y23}), respectively.
\begin{conjecture}\label{con1}
Let a semigroup $S_m\!=\!\la d_1,\ldots,d_m\ra$ be given and $G_r$ denote its 
genera  according to (\ref{y16}). Then the alternating power sums ${\mathbb C}
_k(S_m)$ in (\ref{y5}) are given by (\ref{y22}) with $K_p$ as follows
\bea
K_p=\sum_{r=0}^p{p\choose r}T_{p-r}(\sigma)G_r+\frac{2^{p+1}}{p+1}T_{p+1}
(\delta),\label{y27}
\eea
\end{conjecture}
If (\ref{y27}) holds for any $p$, then, combining it with (\ref{y26}) and 
keeping in mind $\sigma_i,\delta_i>0$, we get $K_p>0$.
\section{Algebraic equations for $K_r$ in numerical semigroups}\label{l4}
Consider a numerical semigroup $S_m$ and write equalities (\ref{y6}) and 
(\ref{y22}) for alternating sums ${\mathbb C}_k(S_m)$, defined in (\ref{y5}), 
as a system of non-homogeneous polynomial equations for positive integer 
variables $C_{k,j}$. For convenience, rename $C_{k,j}$ by one-index variable 
$z_i$ in such a way, that $i$ runs through the two-index $(k,j)$ table, 
enumerating elements of the 1st and following rows successively,
\bea
\quad
z_1\!=\!C_{1,1},\;\;z_2\!=\!C_{1,2},\;\;\ldots,\;\;z_{\beta_1}\!=\!C_{1,\beta_1}
,\;\;z_{\beta_1+1}\!=\!C_{2,1},\;\;z_{\beta_1+2}\!=\!C_{2,2},\;\;\ldots,\;\;
z_{\zeta_m}\!=\!C_{m-1,\beta_{m-1}},\nonumber
\eea
where a total number $\zeta_m\!=\!\#\{z_i\}$ of independent variables $z_i$ is 
dependent on inner properties of $S_m$, 

\noindent
e.g., non-symmetric, symmetric (not CI), symmetric CI and others more 
sophisticated (Weierstrass', Arf's, hyperelliptic {\em etc}). We consider here 
the three basic kind of semigroups mentioned above.

A study of $n+1$ {\it non-homogeneous} polynomial equations $f_j(z_1,\ldots,z_n)
=0$ in $n$ variables $z_i$ goes back to classical works of B\'ezout, Sylvester, 
Caley and Macaulay \cite{ma02}, who has mostly considered an equivalent problem 
with $n+1$ {\it homogeneous} polynomial equations in $n+1$ variables. The use of
a multivariate resultant ${\mathcal Res}\left\{f_0,f_1,\ldots,f_n\right\}$, 
which is an irreducible polynomial over a ring $A[f_0,f_1,\ldots,f_n]$, 
generated by coefficients of $f_j$, and vanishes whenever all polynomials $f_j$ 
have a common root, is a standard computational tool in the elimination theory. 
An interest in finding explicit formulas for resultants, extending Macaulay's 
formulas as a quotient of two determinants, has been renewed in the 1990th 
(see \cite{gkz4} and references therein). 

Bearing in mind a special form of equations (\ref{y6}) and (\ref{y22}), we 
consider here the most general properties of these equations: the existence of 
an algebraic relation among $K_p$, which entered in (\ref{y22}), and its 
rescaled version, given below. Rewrite polynomial equations (\ref{y6}) and 
(\ref{y22}) in new notations,
\bea
&&\Gamma_k(z_1,\ldots,z_{\zeta_m},L_k)=0,\quad k\ge 1,\qquad\zeta_m=B_m,
\label{y28}\\
&&L_k=\left\{\begin{array}{lll}0,&if&1\le k\le m-2,\\(-1)^m(m-1)!\pi_m,&if&k=m
-1,\\(-1)^m\frac{k!}{(k-m)\:!}\;\pi_mK_{k-m},&if&k\ge m,\end{array}\right.
\nonumber
\eea
where $\Gamma_k(z_1,\ldots,z_{\zeta_m},L_k)$ is a homogeneous polynomial of 
degree $k$ with respect to all variables $z_i$ and linear in $L_k$, and $B_m=
\sum_{k=1}^{m-1}\beta_k$ denotes the total Betti number of non-symmetric 
semigroups. 

Making use of a homogeniety of the polynomial $\Gamma_k(\xi_1,\ldots,\xi_{
\zeta_m},\ell_k)$, rescale the variables and the whole equation (\ref{y28}) as 
follows,
\bea
&&\Gamma_k(\xi_1,\ldots,\xi_{\zeta_m},\ell_k)=0,\quad k\ge 1,\qquad\xi_j=
\frac{z_j}{\upsilon_m},\quad\upsilon_m=\pi_m^{1/(m-1)},\label{y29}\\
&&\ell_k=\left\{\begin{array}{lll}0,&if&1\le k\le m-2,\\(-1)^m(m-1)!&if&k=m-1,
\\(-1)^m\frac{k!}{(k-m)\:!}\;\varkappa_{k-m},&if&k\ge m,\qquad\varkappa_p=K_p
\upsilon_m^{-(p+1)}\end{array}\right.\label{y30}
\eea
where according to (\ref{y5}), the polynomial $\Gamma_k(\xi_1,\ldots,\xi_{
\zeta_m},\ell_k)$ for the arbitrary non-symmetric semigroup $S_m$ reads,
\bea
\Gamma_k(\xi_1,\ldots,\xi_{\zeta_m},\ell_k)=\sum_{j=1}^{\beta_1}\xi_j^k-
\sum_{j=\beta_1+1}^{\beta_1+\beta_2}\xi_j^k+\ldots+(-1)^m\!\!\!\sum_{j=\zeta_m-
\beta_{m-1}}^{\zeta_m}\!\!\!\xi_j^k-\ell_k.\quad\label{y31}
\eea
\begin{theorem}\label{the1}
Let $S_m$ be a non-symmetric semigroup with the Hilbert series given in 
(\ref{y1}). Then there exists an algebraic equation in $g_m+1$ variables, 
$\varkappa_0,\varkappa_1,\ldots,\varkappa_{g_m}$,
\bea
{\mathcal R}_K(\varkappa_0,\varkappa_1,\ldots,\varkappa_{g_m-1},\varkappa_{g_m})
=0,\qquad g_m=B_m-m+1,\label{y32}
\eea
and the polynomial ${\mathcal R}_K$ is irreducible over a ring $A[\varkappa_0,
\varkappa_1,\ldots,\varkappa_{g_m-1},\varkappa_{g_m}]$.
\end{theorem}
\begin{proof}
Choose
\footnote{Note that inequality $B_m>m-1$ holds, since, according to \cite{rs09},
we have $\beta_1\ge m-1$ while the other $\beta_k$ are positive.}
the first $B_m+1$ polynomial equations (\ref{y29}) in $B_m$ variables $\xi_1,
\ldots,\xi_{B_m}$ and build a new system of $B_m$ equations in $B_m-1$ variables
$\xi_2,\ldots,\xi_{B_m}$ by eliminating $\xi_1$ in resultants ${\mathcal Res}_1
\left\{\Gamma_1,\Gamma_j\right\}$,
\bea
&&\Gamma_j^1(\xi_2,\ldots,\xi_{B_m},\ell_1,\ell_j)=0,\qquad j=2,\ldots,B_m+1,
\label{y33}\\
&&\Gamma_j^1(\xi_2,\ldots,\xi_{B_m},\ell_1,\ell_j)={\mathcal Res}_1\left\{
\Gamma_1(\xi_1,\ldots,\xi_{B_m},\ell_1),\Gamma_j(\xi_1,\ldots,\xi_{B_m},\ell_j)
\right\}.\nonumber
\eea
The polynomial $\Gamma_j^1$ in (\ref{y33}) is irreducible \cite{gkz4} over a 
ring $A[\xi_2,\ldots,\xi_{B_m},\ell_1,\ell_j]$ (see a detailed proof of a 
resultant irreducibility for two polynomials in \cite{hej18}). 

At the 2nd step, choose the first $B_m$ polynomial equations (\ref{y33}) in 
$B_m-1$ variables $\xi_2,\ldots,\xi_{B_m}$ and build $B_m-1$ equations in $B_m-
2$ variables $\xi_3,\ldots,\xi_{B_m}$ by eliminating $\xi_2$ in resultants 
${\mathcal Res}_2\left\{\Gamma_j^1,\Gamma_2^1\right\}$,
\bea
&&\Gamma_j^{1,2}(\xi_3,\ldots,\xi_{B_m},\ell_1,\ell_2,\ell_j)=0,\qquad j=3,
\ldots,B_m+1,\nonumber\\
&&\Gamma_j^{1,2}(\xi_3,\ldots,\xi_{B_m},\ell_1,\ell_2,\ell_j)={\mathcal Res}_2
\left\{\Gamma_2^1(\xi_2,\ldots,\xi_{B_m},\ell_1,\ell_2),\Gamma_j^1(\xi_2,\ldots,
\xi_{B_m},\ell_1,\ell_j)\right\}.\nonumber
\eea
The polynomial $\Gamma_j^{1,2}$ is irreducible \cite{gkz4} over a ring $A[\xi_3,
\ldots,\xi_{B_m},\ell_1,\ell_2,\ell_j]$ by reasons mentioned above. 

Continuing to eliminate the variables $\xi_k$ successively and constructing the 
families of resultants, 
\bea
&&\Gamma_j^{1,\ldots,k}(\xi_{k+1},\ldots,\xi_{B_m},\ell_1,\ldots,\ell_k,\ell_j)
=0,\qquad j=k+1,\ldots,B_m+1,\nonumber\\
&&\Gamma_j^{1,\ldots,k}(\xi_{k+1},\ldots,\xi_{B_m},\ell_1,\ldots,\ell_k,\ell_j)
={\mathcal Res}_k\left\{\Gamma_k^{1,\ldots,k-1},\Gamma_j^{1,\ldots,k-1}\right\},
\nonumber
\eea
we arrive at the $B_m$th step at one resultant equation
{\small
\bea
{\mathcal Res}_{B_m}\left\{\Gamma_{B_m}^{1,\ldots,B_m-1}\left(\xi_{B_m},\ell_1,
\ldots,\ell_{B_m-1},\ell_{B_m}\right),\Gamma_{B_m+1}^{1,\ldots,B_m-1}\left(
\xi_{B_m},\ell_1,\ldots,\ell_{B_m-1},\ell_{B_m+1}\right)\right\}=0.\label{y34}
\eea}
The polynomial ${\mathcal Res}_{B_m}$ in the l.h.s. of (\ref{y34}) is 
irreducible \cite{gkz4} over a ring $A[\ell_1,\ldots,\ell_{B_m-1},\ell_{B_m},
\ell_{B_m+1}]$ as well as all resultants of two polynomials at previous steps.

Equation (\ref{y34}) is free of any variable $\xi_i$ and involves only $B_m+1$
coefficients $\ell_n$, $1\le n\le B_m+1$. Keeping in mind the two first 
relations in (\ref{y30}), namely, $\ell_n=0$ if $1\le n<m-1$, and an 
independence of $\ell_{m-1}$ on $\varkappa_j$, we conclude that equation 
(\ref{y34}) is algebraic in $B_m-m+2$ variables $\ell_m,\ell_{m+1},\ldots,\ell_{
B_m},\ell_{B_m+1}$. However, by the 3rd relation in (\ref{y30}), such equation 
can be represented in $\varkappa_0,\varkappa_1,\ldots,\varkappa_{B_m-m},
\varkappa_{B_m-m+1}$ as given in (\ref{y32}).$\;\;\;\;\;\;\Box$
\end{proof}
\begin{corollary}\label{cor1}
Let $S_m$ be a non-symmetric semigroup with the Hilbert series given in 
(\ref{y1}). Then there exists $g_m$ algebraically independent genera. The set 
of such genera reads,
\bea
\{G_0,G_1,\ldots,G_{g_m-1}\}.\nonumber
\eea
\end{corollary}
\begin{proof}
Combining Theorem \ref{the1} and formulas (\ref{y23}) as well as (\ref{y27}), by
assumption that Conjecture \ref{con1} is true, we arrive at algebraic equation
\bea
{\mathcal R}_G(G_0,G_1,\ldots,G_{g_m-1},G_{g_m})=0,\nonumber
\eea
with an irreducible polynomial ${\mathcal R}_G$ over a ring $A[G_0,G_1,\ldots,G_
{g_m-1},G_{g_m}]$. Resolving the last equation with respect to $G_{g_m}$ and 
keeping in mind an irreduciblity of ${\mathcal R}_G$, we arrive at the algebraic
function $G_{g_m}={\cal F}(G_0,\ldots,G_{g_m-1})$, where the set $\{G_0,G_1,
\ldots,G_{g_m-1}\}$ comprises genera for any numerical semigroup $S_m$, which 
are algebraically independent.$\;\;\;\;\;\;\Box$
\end{proof}
Theorem \ref{the1} may be extended on algebraic equations included $\varkappa_
n$, $n>g_m$, with a similar proof.
\begin{theorem}\label{the2}
There exists an algebraic equation in $g_m+1$ variables $\varkappa_0,\varkappa_1
,\ldots,\varkappa_{g_m-1}$ and $\varkappa_n$,
\bea
{\mathcal R}_K(\varkappa_0,\varkappa_1,\ldots,\varkappa_{g_m-1},\varkappa_n)=0,
\qquad n>g_m,\label{y35}
\eea
and the polynomial ${\mathcal R}_K$ is irreducible over a ring $A[\varkappa_0,
\varkappa_1,\ldots,\varkappa_{g_m-1},\varkappa_n]$.
\end{theorem}
\subsection{Numerical semigroups $S_3$}\label{l41}
In this section we consider the most simple case of non-symmetric numerical 
semigroups $S_3$ generated by three integers. The numerator in the rational Rep 
(\ref{y3}) of its Hilbert series $H(S_3;z)$ reads,
\bea
Q(S_3;z)=1-\left(z^{x_1}+z^{x_2}+z^{x_3}\right)+z^{y_1}+z^{y_2},\qquad\beta_1=3,
\;\;\beta_2=2,\;\;g_3=3.\nonumber
\eea
Six polynomial equations (\ref{y6}) and (\ref{y22}) for five symmetric
polynomials, $X_k=\sum_{j=1}^3x_j^k$, $k=1,2,3$, and $Y_r=\sum_{j=1}^2y_j^r$, 
$r=1,2$, are given below,
\bea
\left.\begin{array}{lllll}
Y_1=X_1,&&Y_2=X_2+2\pi_3,&&Y_3=X_3+6\pi_3K_0,\\
Y_4=X_4+24\pi_3K_1,&&Y_5=X_5+60\pi_3K_2,&&Y_6=X_6+120\pi_3K_3,\end{array}
\right.\label{y36}
\eea
where $K_i$ are given in (\ref{y23}). Bearing in mind the Newton identities 
\cite{mc95} related symmetric polynomials,
\bea
&&Y_3=\frac1{2}\left(3Y_2-Y_1^2\right)Y_1,\qquad 
Y_4=\frac1{2}\left(Y_2^2+2Y_1^2Y_2-Y_1^4\right),\label{y37}\\
&&Y_5=\frac1{4}\left(5Y_2^2-Y_1^4\right)Y_1,\qquad
Y_6=\frac1{4}\left(Y_2^2+6Y_1^2Y_2-3Y_1^4\right)Y_2,\nonumber\\
&&X_4=\frac1{6}\left(X_1^4-6X_1^2X_2+8X_1X_3+3X_2^2\right),\qquad X_5=\frac1{6}
\left(X_1^5-5X_1^3X_2+5X_1^2X_3+5X_2X_3\right),\nonumber\\
&&X_6=\frac1{12}\left(X_1^6-3X_1^4X_2-9X_1^2X_2^2+3X_2^3+4X_1^3X_3+12X_1X_2
X_3+4X_3^2\right),\nonumber
\eea
we present six equations (\ref{y36}) as follows
\bea
\frac1{2}Y_1\left(3Y_2-Y_1^2\right)&\!\!=\!\!&X_3+6\pi_3K_0,\qquad 
Y_1=X_1,\qquad Y_2=X_2+2\pi_3,\label{y38}\\
3\left(Y_2^2+2Y_1^2Y_2-Y_1^4\right)&\!\!=\!\!&X_1^4-6X_1^2X_2+8X_1X_3+3X_2^2+
144\pi_3K_1,\nonumber\\
\frac{3}{2}Y_1\left(5Y_2^2-Y_1^4\right)&\!\!=\!\!&X_1^5-5X_1^3X_2+5X_1^2X_3+
5X_2X_3+360\pi_3K_2,\nonumber\\
3Y_2\left(Y_2^2+6Y_1^2Y_2-3Y_1^4\right)&\!\!=\!\!&X_1^6-3X_1^4X_2-9X_1^2X_2^2+
3X_2^3+4X_1^3X_3+12X_1X_2X_3+\nonumber\\
&&4X_3^2+1440\pi_3K_3.\nonumber
\eea
Substituting three first relations of (\ref{y38}) into the three last and
simplifying the final expressions, we obtain
\bea
&&Y_1^2-4K_0Y_1+12K_1+\pi_3=Y_2,\qquad Y_1^3-2K_0Y_1^2+4\pi_3K_0+24K_2=Y_2(Y_1+
2K_0),\nonumber\\
&&Y_1^4-2\pi_3Y_1^2+8\pi_3K_0Y_1+8\pi_3K_0^2-\frac{4}{3}\pi_3^2 +80K_3=Y_2^2-
2Y_2(\pi_3-4K_0Y_1).\label{y39}
\eea
Combining separately the 1st relation in (\ref{y39}) with the 2nd and 3rd
relations, we get, respectively,
\bea
\left(3K_1-2K_0^2+\frac{\pi_3}{4}\right)Y_1&=&6K_2-6K_0K_1+\frac{\pi_3}{2}
K_0,\label{y40}\\
\left(3K_1-2K_0^2+\frac{\pi_3}{4}\right)Y_1^2&=&10K_3-18K_1^2+\pi_3K_0^2-
\frac{\pi_3^2}{24},\nonumber
\eea
and further, due to (\ref{y36},\ref{y37},\ref{y39}),
\bea
X_2&=&4\left(\frac{3K_2-6K_0K_1+2K_0^3}{3K_1-2K_0^2+\frac{\pi_3}{4}}\right)^2-
4\left(K_0^2-3K_1\right)-\pi_3,\qquad X_1=Y_1,\nonumber\\
X_3&=&Y_1^3-6K_0Y_1^2+3\left(6K_1+\frac{\pi_3}{2}\right)Y_1-6K_0\pi_3,\qquad 
Y_2=X_2+2\pi_3.\nonumber
\eea
Rescaling $K_r=\varkappa_r\sqrt{\pi_3^{r+1}}$, write a necessary condition to 
have non-trivial solutions for equations (\ref{y40}),
\vspace{-.2cm}
$$
3\varkappa_1\ne 2\varkappa_0^2-\frac1{4},
$$
otherwise, there exist three equalities
\bea
\varkappa_1=\frac{2}{3}\varkappa_0^2-\frac1{12},\qquad\varkappa_2=\varkappa_0
\left(\varkappa_1-\frac1{12}\right),\qquad\varkappa_3=\frac{9}{5}\varkappa_1^2-
\frac1{10}\varkappa_0^2+\frac1{240},\label{y41}
\eea
which define a special class of semigroups $S_3$. In section \ref{l52}, we show 
that the last three formulas are related to symmetric 3-generated semigroups.

Combining two equalities (\ref{y40}), we obtain, in accordance with Theorem 
\ref{the1}, equation (\ref{y32}) in rescaled variables $\varkappa_0,\varkappa_1,
\varkappa_2,\varkappa_3$,
\bea
\left(10\varkappa_3-18\varkappa_1^2+\varkappa_0^2-\frac1{24}\right)\left(3
\varkappa_1-2\varkappa_0^2+\frac1{4}\right)=\left(6\varkappa_2-6\varkappa_0
\varkappa_1+\frac{\varkappa_0}{2}\right)^2,\label{y42}
\eea
that manifests three algebraically independent genera $G_0,G_1,G_2$.
\subsubsection{Extension on higher $\varkappa_n$ in $S_3$}\label{411}
To derive the equation (\ref{y35}) for $\varkappa_4$ let us consider the power
sums $X_7,Y_7$,
\bea
X_7\!=\!\frac{X_1^7\!-\!21X_1^3X_2^2+7X_1^4X_3+21X_2^2X_3+28X_1X_3^2}{36},\qquad
Y_7\!=\!\frac{(Y_1^6\!-\!7Y_1^4Y_2+7Y_1^2Y_2^2+7Y_2^3)Y_1}{8},\nonumber
\eea
and substitute them into equality $Y_7\!=\!X_7+840\pi_3K_4$. Making use of the
last equality and five first relations in (\ref{y38}) and performing necessary
calculations, we arrive at three equations for $Y_1$ and $Y_2$.
\bea
Y_1^2-4K_0Y_1+12K_1+\pi_3=Y_2,\qquad Y_1^3-2K_0Y_1^2+4\pi_3K_0+24K_2=
Y_2(Y_1+2K_0),\nonumber\\
8K_0^2\pi_3Y_1+Y_1(Y_1^2-Y_2)(Y_2-\pi_3)+K_0(Y_1^4-4\pi_3^2+4\pi_3Y_2-4Y_1^2Y_2
-Y_2^2)=-60K_4,\label{y43}
\eea
which is similar to (\ref{y39}) by exception of the last one. Equations 
(\ref{y43}) can be resolved in $Y_1$ as follows,
\bea
\left(3K_1-2K_0^2+\frac{\pi_3}{4}\right)Y_1&=&6K_2-6K_0K_1+\frac{\pi_3}{2}K_0,
\nonumber\\
\left(Y_1^2-2K_0Y_1+12K_1\right)\left(3K_1-2K_0^2+\frac{\pi_3}{4}\right)
Y_1&=&15K_4-K_0\left(\frac{\pi_3}{2}-6K_1\right)^2,\label{y44}
\eea
Combining two equations in (\ref{y44}) and rescaling the coefficients $K_r=
\varkappa_r\sqrt{\pi_3^{r+1}}$, we obtain, in accordance with Theorem 
\ref{the1}, equation (\ref{y35}) in $\varkappa_0,\varkappa_1,\varkappa_2,
\varkappa_4$ variables,
\bea
\left(15\varkappa_4-\varkappa_0\left(6\varkappa_1-\frac1{2}\right)^2\right)   
\left(3\varkappa_1-2\varkappa_0^2+\frac1{4}\right)^2=\left(6\varkappa_2-6
\varkappa_1\varkappa_0+\frac{\varkappa_0}{2}\right){\mathcal K}_4(\varkappa_0,
\varkappa_1,\varkappa_2),\;\;\label{y45}
\eea
where ${\mathcal K}_4(\varkappa_0,\varkappa_1,\varkappa_2)$ is a positive 
definite function
\bea
{\mathcal K}_4=2\left(1\!+\!12\varkappa_1\right)\varkappa_0^4+24\varkappa_2
\varkappa_0^3-18\varkappa_1\left(1\!+\!4\varkappa_1\right)\varkappa_0^2+3
\varkappa_2\left(1\!-\!36\varkappa_1\right)\varkappa_0+36\varkappa_2^2+\frac{3
\varkappa_1}{4}\left(1\!+\!12\varkappa_1\right)^2\nonumber
\eea
in the positive octant $\varkappa_0,\varkappa_1,\varkappa_2>0$. In order to
prove that, we suppose, by way of contradiction, that ${\mathcal K}_4(\varkappa
_0,\varkappa_1,\varkappa_2)=0$. The last equation may be resolved as quadratic
in $\varkappa_2$,
\bea
\varkappa_2^{\pm}=\pm\left(1-8\varkappa_0^2+12\varkappa_1\right)\sqrt{\varkappa
_0^2-12\varkappa_1}+\varkappa_0\left(36\varkappa_1-1-8\varkappa_0^2\right),
\quad\varkappa_2^-\le\varkappa_2^+.\label{y46}
\eea
Consider the largest real root $\varkappa_2^+\ge \varkappa_2^-$ and require 
$12\varkappa_1\le\varkappa_0^2$. Combining (\ref{y46}) with the last inequality,
we arrive for $\varkappa_0,\varkappa_1>0$ at the upper bound,
\bea
\varkappa_2^+\le\left(1-7\varkappa_0^2\right)\sqrt{\varkappa_0^2-12\varkappa_1}
-\varkappa_0\left(1+5\varkappa_0^2\right)\le\varkappa_0\left(1-7\varkappa_0^2-
1-5\varkappa_0^2\right)\le -12\varkappa_0^3.\nonumber
\eea
Thus, the both roots $\varkappa_2^\pm$ are never positive. In other words, in
the positive octant $\varkappa_0,\varkappa_1,\varkappa_2>0$ the function
${\mathcal K}_4(\varkappa_0,\varkappa_1,\varkappa_2)$ is never vanished. Since  
${\mathcal K}_4(\varkappa_0,0,0)=2\varkappa_0^4$, we conclude that the function
${\mathcal K}_4(\varkappa_0,\varkappa_1,\varkappa_2)$ is always positive.

Coming back to relation (\ref{y44}) and equations (\ref{y45}), we conclude that
there exists one special case, when all (\ref{y44},\ref{y45}) are satisfied
identically,
\bea
3\varkappa_1-2\varkappa_0^2+\frac1{4}=0,\qquad 6\varkappa_2-6\varkappa_1
\varkappa_0+\frac{\varkappa_0}{2}=0,\qquad 15\varkappa_4-\varkappa_0\left(
6\varkappa_1-\frac1{2}\right)^2=0,\label{y47}
\eea
This case is related to symmetric 3-generated semigroups and equalities  
(\ref{y47}) are coincided with the three corresponding formulas for $K_1,K_2,
K_4$ in (\ref{y67}).
\section{Supplementary relations for $K_r$ and $G_r$ in symmetric semigroups}
\label{l5}
If a numerical semigroup $S_m$ is symmetric, then degrees $C_{k,j}$ of syzygies 
and Betti's numbers $\beta_k$ in the rational Rep (\ref{y2}) of the Hilbert 
series are related as follows,
\bea
\beta_k=\beta_{m-k-1},\quad\beta_{m-1}=1,\qquad C_{k,j}+C_{m-k-1,j}=Q_m,
\label{y48}
\eea
while the number of gaps and non-gaps of $S_m$ are equal to $G_0$. Therefore,
according to (\ref{y4},\ref{y23}), we have 
\bea
F_m=2G_0-1,\quad 2K_0=Q_m=F_m+\sigma_1,\qquad \sum_{j=1}^{G_0}s_j^n+\sum_{j=1}
^{G_0}(F_m-s_j)^n=\sum_{j=0}^{F_m}j^n.\label{y49}
\eea
The last identity in (\ref{y49}) may be represented as follows
\bea
2G_{2r}+\sum_{q=1}^{2r}(-1)^q{2r\choose q}F_m^qG_{2r-q}=\frac{F_m^{2r+1}}{2r+1}+
\frac{F_m^{2r}}{2}+\sum_{q=1}^{2r-1}{2r\choose q-1}{\cal B}_{2r-q+1}\frac{F_m^q}
{q},\label{y50}
\eea
where ${\cal B}_k$ denotes the Bernoulli number. Equality (\ref{y50}) reduces 
the number of independent genera $G_k$ twice, making $G_{2r}$ dependent on 
genera with odd indices, $G_{2j-1}$, $j=1,\ldots,r$,
\bea
\frac{G_{2r}}{F_m}=\frac1{2}\sum_{p=0}^{2r-2}\left[{2r\choose p}\frac{{\cal B}_
{2r-p}}{p+1}+(-1)^p{2r\choose p+1}G_{2r-p-1}\right]F_m^p-\frac{2r-1}{2r+1}\frac{
F_m^{2r}}{4},\label{y51}
\eea
e.g.,
\bea
\frac{G_2}{F_m}&=&G_1-\frac{F_m^2-1}{12},\hspace{4cm}\frac{F_m^2-1}{12}=
\frac{G_0(G_0-1)}{3},\label{y52}\\
\frac{G_4}{F_m}&=&2G_3-F_m^2G_1+\frac{F_m^2-1}{12}\;\frac{6F_m^2+1}{5},
\nonumber\\
\frac{G_6}{F_m}&=&3G_5-5F_m^2G_3+3F_m^4G_1-\frac{F_m^2-1}{12}\;\frac{51F_m^4+
9F_m^2+2}{14},\nonumber\\
\frac{G_8}{F_m}&=&4G_7-14F_m^2G_5+28F_m^4G_3-17F_m^6G_1+\frac{F_m^2-1}{12}\;
\frac{310F_m^6+55F_m^4+13F_m^2+3}{15}.\nonumber
\eea

To find a number $g_m$ of algebraically independent genera of symmetric 
semigroups, we apply Theorem \ref{the1} with a new number $\zeta_m$ of 
independent variables, which differs from (\ref{y31}). For this purpose, 
represent formula (\ref{y3}) for the numerator $Q\left(S_m;t\right)$ when $m=0,
1(\bmod\;2)$ separately and account for independent variables $x_j,y_j,\ldots,
z_j,Q_m$ in every of two cases,
\begin{itemize}
\item $m=2q,\;\;q\ge 2$,\hspace{.7cm} $\zeta_{2q}=\sum_{j=0}^{q-1}\beta_j$
\bea
1\!-\!t^{Q_m}\!-\sum_{j=1}^{\beta_1}\left(t^{x_j}-t^{Q_m-x_j}\right)+\sum_{j=1}^
{\beta_2}\left(t^{y_j}-t^{Q_m-y_j}\right)-\!\ldots\!-(-1)^q\sum_{j=1}^{\beta_
{q-1}}\left(t^{z_j}-t^{Q_m-z_j}\right),\quad\label{y53}
\eea
\item $m=2q+1\;\;q\ge 2$,\hspace{.7cm}$\zeta_{2q+1}=\sum_{j=0}^{q}\beta_j$
\hspace{.7cm} $(-1)^q\beta_q=\sum_{j=0}^{q-1}(-1)^{j+1}\beta_j$
\bea
1\!+\!t^{Q_m}\!-\sum_{j=1}^{\beta_1}\left(t^{x_j}+t^{Q_m-x_j}\right)+\sum_{j=1}
^{\beta_2}\left(t^{y_j}+t^{Q_m-y_j}\right)-\!\ldots\!+(-1)^q\sum_{j=1}^{\beta
_q}\left(t^{z_j}+t^{Q_m-z_j}\right).\quad\label{y54}
\eea
\end{itemize}
Note, that $\zeta_{2q+1}=0\;(\bmod\;2)$ since $\zeta_{2q+1}$ may be presented
for $m=1,3(\bmod\;4)$ as follows
\bea
\zeta_{4q+1}=2\sum_{j=1}^q\beta_{2j-1},\qquad 
\zeta_{4q+3}=2\left(\sum_{j=1}^q\beta_{2j}+1\right).\label{y55}
\eea
\begin{theorem}\label{the3}
Let $S_m$ be a symmetric (not CI) semigroup, then there are $g_m$ independent 
genera
\bea
g_{2q}=\zeta_{2q}-q,&& G_0,G_1,G_3,\ldots G_{2g_{2q}-3},\label{y56}\\
g_{2q+1}=\zeta_{2q+1}-q,&& G_0,G_1,G_3,\ldots G_{2g_{2q+1}-3}.\label{y57}
\eea
\end{theorem}
\begin{proof}
First, consider symmetric semigroups $S_{2q}$ with even {\em edim}. The total 
number of syzygies degrees (including those which are related in couples), 
appeared in (\ref{y53}), is given by $2\zeta_{2q}-1$. Replacing by this number 
the total Betti number $B_m$ in (\ref{y32}), we get ${\widetilde g}_{2q}=2
\zeta_{2q}-1-(2q-1)=2(\zeta_{2q}-q)$, which does not related to equalities  
(\ref{y51}). Keeping in mind the supplementary relations (\ref{y51}) for genera 
$G_k$, we have to decrease the last number twice, i.e., we arrive at 
(\ref{y56}).

Next, consider symmetric semigroups $S_{2q+1}$ with odd {\em edim}. By similar 
considerations, as in the case $m=2q$, we arrive at ${\widetilde g}_{2q+1}=2
\zeta_{2q+1}-1-(2q+1-1)=2(\zeta_{2q+1}-q)-1$, which does not related to 
equalities (\ref{y51}). Keeping in mind the supplementary relations (\ref{y51}) 
for genera $G_k$, we have to decrease the number ${\widetilde g}_{2q+1}$ as 
follows: $g_{2q+1}=(1+{\widetilde g}_{2q+1})/2$, i.e., we arrive at (\ref{y57}).
$\;\;\;\;\;\;\Box$
\end{proof}
Combining Theorem \ref{the3} with (\ref{y53},\ref{y55}) we may specify the 
number $g_m$ in more details,
\bea
g_{2q}=\sum_{j=1}^{q-1}\beta_j-q+1,\qquad
g_{4q+1}=2\left(\sum_{j=1}^q\beta_{2j-1}-q\right),\qquad
g_{4q+3}=2\left(\sum_{j=1}^q\beta_{2j}-q\right)+1.\label{y58}
\eea
\subsection{Symmetric (not CI) numerical semigroups $S_4$}\label{l51}
The numerator (\ref{y3}) in the rational Rep of its Hilbert series $H(S_4;z)$ 
reads \cite{be752},
\bea
Q(S_4;z)=1-\sum_{j=1}^5z^{x_j}+\sum_{j=1}^5z^{Q_4-x_j}-z^{Q_4},\qquad\beta_1=5,
\quad g_4=4.\nonumber
\eea
We present polynomial equations (\ref{y6},\ref{y22}) for the ten first symmetric
polynomials, $X_k=\!\sum_{j=1}^5x_j^k$. Among them, equations of the 1st 
and 2nd degrees are coincided. Together with (\ref{y49}) they give
\bea
X_1=2Q_4=4K_0.\label{y59}
\eea
The rest of eight equations might be decomposed in couples of the odd and even 
degrees,
\bea
2X_3-3Q_4X_2+2Q_4^3\!\!\!&=&\!\!\!3!\pi_4,\nonumber\\
2X_3-3Q_4X_2+2Q_4^3\!\!\!&=&\!\!\!\frac{4!}{0!}\frac{K_0\pi_4}{2Q_4}
,\nonumber\\
2X_5-5Q_4X_4+10Q_4^2X_3-10Q_4^3X_2+6Q_4^5\!\!\!&=&\!\!\!
\frac{5!}{1!}K_1\pi_4,\nonumber\\
2X_5-5Q_4X_4+\frac{20}{3}Q_4^2X_3-5Q_4^3X_2+\frac{8}{3}Q_4^5\!\!\!&=&\!\!\!
\frac{6!}{2!}\frac{K_2\pi_4}{3Q_4},\nonumber\\
2X_7-7Q_4X_6+21Q_4^2X_5-35Q_4^3X_4+35Q_4^4X_3-21Q_4^5X_2+10Q_4^7\!\!\!&=&\!\!\!
\frac{7!}{3!}K_3\pi_4,\nonumber\\
2X_7-7Q_4X_6+14Q_4^2X_5-\frac{35}{2}Q_4^3X_4+14Q_4^4X_3-7Q_4^5X_2+3Q_4^7\!\!\!&
=&\!\!\!\frac{8!}{4!}\frac{K_4\pi_4}{4Q_4},\nonumber
\eea
\bea
2X_9\!-\!9Q_4X_8\!+\!36Q_4^2X_7\!-\!84Q_4^3X_6\!+\!126Q_4^4X_5\!-\!126Q_4^5X_4\!
+\!84Q_4^6X_3\!-\!36Q_4^7X_2\!+\!14Q_4^9\!\!\!&=&\!\!\!\frac{9!}{5!}K_5\pi_4,
\nonumber\\
2X_9\!-\!9Q_4X_8\!+\!24Q_4^2X_7\!-\!42Q_4^3X_6\!+\!\frac{252}{5}Q_4^4X_5\!-\!
42Q_4^5X_4\!+\!24Q_4^6X_3\!-\!9Q_4^7X_2\!+\!\frac{16}{5}Q_4^9\!\!\!&=&\!\!\!\!
\frac{10!}{6!}\frac{K_6\pi_4}{5Q_4}\nonumber
\eea
Their successive solution gives,
\bea
&&K_2=2\left(K_1-\frac1{3}K_0^2\right)K_0,\qquad K_4=4\left(K_3-2K_0^2K_1+
\frac{4}{5}K_0^4\right)K_0,\nonumber\\
&&K_6=6\left(K_5-\frac{20}{3}K_0^2K_3+16K_0^4K_1-\frac{136}{21}K_0^6\right)K_0,
\label{y60}
\eea
where $K_0,K_1,K_3,K_5$ are four independent coefficients, in accordance with 
(\ref{y58}). Formulas (\ref{y60}) and (\ref{y23},\ref{y52}) are strongly 
related. Namely, the former may be obtained by a straightforward substitution 
of (\ref{y52}) into (\ref{y23}).

The list of formulas (\ref{y60}) may be continued if we consider equations 
(\ref{y22}) for higher degrees,
\bea
K_{2r}\!=\!2r\left(K_{2r-1}\!-\!\rho_{2r}^1K_0^2K_{2r-3}\!+\!\ldots\!-\!(-1)^r
\rho_{2r}^{r-1}K_0^{2(r-1)}K_1\!+\!(-1)^r\rho_{2r}^rK_0^{2r}\right)K_0,\quad
\rho_{2r}^j\in{\mathbb Q},\nonumber
\eea
where $K_{2j+1}$, $j\ge 3$, are algebraic (not polynomial) functions of 
$K_0,K_1,K_3,K_5$.
\subsection{Supplementary relations for $K_r$ and $G_r$ in symmetric CI 
semigroups}\label{l52}
This kind of numerical semigroups is described by a simple Hilbert series 
(\ref{y2}) with a numerator $Q\left(S_m;z\right)$
\bea
Q\left(S_m;z\right)=\left(1-z^{e_1}\right)\left(1-z^{e_2}\right)\cdots
\left(1-z^{e_{m-1}}\right),\nonumber
\eea
built on $m-1$ degrees $e_j$. The alternating power sum ${\mathbb C}_k(S_m)$ 
reads,
\bea
{\mathbb C}_n(S_m)=\sum_{j=1}^{m-1}e_j^n-\sum_{j>r=1}^{m-1}(e_j+e_r)^n+\ldots-
(-1)^{m-1}\left(\sum_{j=1}^{m-1}e_j\right)^n.\label{y61}
\eea
Then, according to the Rep (\ref{y24},\ref{y25}), the expression in (\ref{y61}) 
may be written as follows,
\bea
{\mathbb C}_n(S_m)=\frac{(-1)^m\;n!}{(n-m+1)!}\varepsilon_{m-1}T_{n-m+1}(E),
\qquad{\mathbb C}_{m-1}(S_m)=(-1)^m(m-1)\:!\:\varepsilon_{m-1},\label{y62}
\eea
where polynomials $T_r(E)=T_r(E_1,\ldots,E_r)$ are built by replacing $X_r\to 
E_r=\sum_{j=1}^{m-1}e_j^r$ in polynomials $T_r(X_1,\ldots,X_r)$, defined in 
(\ref{y25}), and $\varepsilon_{m-1}=\prod_{j=1}^{m-1}e_j$.

Combining formulas (\ref{y22}) and (\ref{y62}), we obtain,
\bea
\pi_m=\varepsilon_{m-1},\qquad K_p=\frac{T_{p+1}(E)}{p+1},\qquad p\ge 0,
\label{y63}
\eea
where the 1st equality was established earlier (see formula (5.4) in 
\cite{fl17}) while the 2nd relation leads to an infinite number of equalities. 
An universality of (\ref{y63}) disappears if we consider symmetric (not CI) 
semigroups, see e.g., formula (\ref{y59}) for $K_0$. The number $g_m$ of 
independent genera $G_r$ is given by the number of degrees of syzygies, bearing 
in mind the 1st relation in (\ref{y63}),
\vspace{-.1cm}
\bea
g_m=m-2\label{y64}
\eea
In symmetric semigroups $S_m$ there holds a strict inequality $\mu>m$ (see 
\cite{fl10}), that bounds a genus from below, $G_0\ge m+1$, and leads, in 
combination with (\ref{y64}), to another inequality $g_m<G_0$.

Below we present three examples with symmetric CI semigroups $S_m$, where $m=2,
3,4$.
\begin{example}\label{ex1}CI semigroup $S_2$, $g_2=0$, $e=\pi_2$,\hspace{1cm}
$E_r=\pi_2^r$,

By formula (17) from \cite{fl20} and (\ref{y63}) we obtain
\bea
T_r(e)=\frac{\pi_2^r}{r+1},\qquad K_r=\frac{\pi_2^{r+1}}{(r+1)(r+2)}.
\label{y65}
\eea
Assuming that Conjecture \ref{con1} is true, substitute the last into 
(\ref{y27}),
\bea
(p+1)\sum_{r=0}^p{p\choose r}T_{p-r}(\sigma)G_r=\frac{\pi_2^{p+1}}{p+2}-
2^{p+1}T_{p+1}(\delta),\nonumber
\eea
which may be resolved with respect to $G_r$ if we make use of the inverse matrix
$\left({p\choose r}T_{p-r}(\sigma)\right)^{-1}$. We give explicit expressions 
for the five first genera $G_r$,
\bea
G_0&=&\frac{\pi_2-\sigma_1+1}{2},\qquad 
G_1=\frac{\pi_2-\sigma_1+1}{2}\;\frac{2\pi_2-\sigma_1-1}{6},\label{y66}\\
G_2&=&\frac{\pi_2-\sigma_1+1}{2}\;\frac{\pi_2(\pi_2-\sigma_1)}{6},\nonumber\\
G_3&=&\frac{\pi_2-\sigma_1+1}{2}\;\frac{6\pi_2^3+\pi_2^2(4-9\sigma_1)+\pi_2
(\sigma_1^2-2\sigma_1-1)+(\sigma_1+1)(\sigma_1^2+1)}{60},\nonumber\\
G_4&=&\frac{\pi_2-\sigma_1+1}{2}\;\frac{\pi_2(\pi_2-\sigma_1)}{6}\;
\frac{2\pi_2^2-\pi_2(2\sigma_1-3)-\sigma_1^2}{5}.\nonumber
\eea
Formulas (\ref{y66}) coincide with expressions for genera, derived \cite{ro94} 
in terms of generators $d_1,d_2$, e.g., 
\bea
G_0=\frac{(d_1-1)(d_2-1)}{2},\quad
G_1=\frac{(d_1-1)(d_2-1)(2d_1d_2-d_1-d_2-1)}{12}.\nonumber
\eea
\end{example}

\begin{example}\label{ex2}CI semigroup $S_3$, $g_3=1$,\hspace{1cm}$E_1=2K_0$.

There exists one independent power sum $E_1$, while the sums $E_{2r}$ may be 
expressed as follows,
\bea
E_2=E_1^2-2\pi_3,\qquad E_4=E_1^4-4\pi_3E_1^2+2\pi_3^2,\qquad E_6=E_1^6-
6\pi_3E_1^4+9\pi_3^2E_1^2-2\pi_3^3.\nonumber
\eea
Substituting $E_{2k}$ into expressions for $T_r(E)$ in Appendix \ref{appendix3} 
and subsequently into (\ref{y63}) we obtain
\bea
K_1&=&\frac1{3}\left(2K_0^2-\frac{\pi_3}{4}\right),\label{y67}\\
K_2&=&\frac1{3}\left(2K_0^2-\frac{\pi_3}{2}\right)K_0,\nonumber\\
K_3&=&\frac1{5}\left(4K_0^4-\frac{3\pi_3}{2}K_0^2+\frac{\pi_3^2}{12}\right),
\nonumber\\
K_4&=&\frac{4}{15}\left(4K_0^4-2\pi_3K_0^2+\frac{\pi_3^2}{4}\right)K_0,\nonumber\\
K_5&=&\frac1{21}\left(32K_0^6-20\pi_3K_0^4+4\pi_3^2K_0^2-\frac{\pi_3^3}{8}
\right),\nonumber\\
K_6&=&\frac1{7}\left(16K_0^6-12\pi_3K_0^4+\frac{10\pi_3^2}{3}K_0^2-\frac{\pi_3
^3}{4}\right)K_0.\nonumber
\eea
Combining (\ref{y67}) with formulas (\ref{y23}) and (\ref{y52}), we obtain for 
genera $G_r$ the polynomial expressions in $G_0$. We present here only the four 
first formulas; expressions for $G_5,G_6$ are extremely lengthy.
\bea
G_1&=&\frac{2}{3}G_0^2+\left(\frac{\delta_1}{3}-\frac1{2}\right)G_0+\frac{
\gamma}{6},\hspace{2cm}\gamma=\delta_1^2-\delta_2-\frac{\pi_3}{2},\label{y68}\\
G_2&=&\frac{2}{3}G_0^3+\frac{2}{3}\left(\delta_1-1\right)G_0^2+\frac1{3}\left(
\gamma-\delta_1+\frac1{2}\right)G_0-\frac{\gamma}{6},\nonumber\\
G_3&=&\frac{4}{5}G_0^4+\left(\frac{6\delta_1}{5}-1\right)G_0^3+\left(\frac{2
\gamma}{3}+\frac{2\delta_1^2}{15}+\frac{\pi_3}{30}-\delta_1+\frac1{3}\right)
G_0^2+\nonumber\\
&&\left[\gamma\left(\frac{\delta_1}{5}-\frac1{2}\right)-\frac{\delta_1}{3}
\left(\frac{2\delta_2}{5}-\frac1{2}\right)\right]G_0+\frac{\delta_1^2(\delta_1^2
-\pi_3)}{20}-\frac{\delta_2^2+2\gamma\delta_2-\gamma}{12}+\frac{\delta_4}{30}+
\frac{\pi_3^2}{60}\nonumber\\
G_4&=&\frac{16}{15}G_0^5+\frac{8}{5}\left(\frac{4\delta_1}{3}-1\right)G_0^4+
2\left(\frac1{3}-\frac{6\delta_1}{5}+\frac{4\delta_1^2+\pi_3}{15}+\frac{2\gamma}
{3}\right)G_0^3+\nonumber\\
&&\left(\frac{2\delta_1}{3}+4\gamma\left(\frac{\delta_1}{5}-\frac1{3}\right)-
\frac{4\delta_1^2+8\delta_1\delta_2+\pi_3}{15}\right)G_0^2+
\left(\frac{3\delta_1^4+4\delta_1\delta_2-5\delta_2^2+2\delta_4}{15}-\right.
\nonumber\\
&&\left.\gamma\frac{6\delta_1+10\delta_2-5}{15}+\frac{\pi_3^2}{15}-\frac1{30}
\right)G_0+\frac{5\delta_2^2-3\delta_1^4-2\delta_4}{15}+\frac{\gamma\delta_2}
{3}+\frac{\pi_3}{10}\left(\delta_1^2-\frac{\pi_3}{3}\right).\nonumber
\eea
In \cite{fr07}, formulas for $G_r$ were derived in terms of 3 diagonal elements 
of matrix of minimal relations for generators $d_j$, that makes them less 
convenient from computational point of view than formulas (\ref{y67},\ref{y68}).
\end{example}
\begin{example}\label{ex3}CI semigroup $S_4$, $g_4=2$,\hspace{1cm}$E_1=2K_0$,
\hspace{.5cm}$E_2=12(2K_1-K_0^2)$.

There exist two independent power sums $E_1,E_2$, while the rest $E_k$ may be
expressed as follows,
\bea
&&2E_3=3E_1E_2-E_1^3+6\pi_4,\nonumber\\
&&2E_4=E_2^2+2E_1^2E_2-E_1^4+8\pi_4E_1,\nonumber\\
&&4E_5=5E_1E_2^2-E_1^5+10\pi_4(E_1^2+E_2),\nonumber\\
&&4E_6=6E_1^2E_2^2+E_2^3-3E_1^4E_2+24\pi_4E_1E_2+12\pi_4^2.\nonumber
\eea
Substituting $E_k$ into expressions for $T_r(E)$ in Appendix \ref{appendix3} 
and subsequently into (\ref{y63}) we obtain
\bea
K_2&=&2\left(K_1-\frac1{3}K_0^2\right)K_0,\label{y69}\\
K_3&=&\frac1{5}\left(12K_1^2+2K_0^2K_1-\frac{8}{3}K_0^4-\frac{\pi_4}{12}K_0
\right),\nonumber\\
K_4&=&\frac1{5}\left(48K_1^2-32K_0^2K_1+\frac{16}{3}K_0^4-\frac{\pi_4}{3}K_0
\right)K_0,\nonumber\\
K_5&=&\frac1{7}\left(72K_1^3+48K_1^2K_0^2-80K_1K_0^4+\frac{64}{3}K_0^6-
\pi_4K_1K_0-\frac{2\pi_4}{3}K_0^3+\frac{\pi_4^2}{72}\right),\nonumber\\
K_6&=&\frac1{7}\left(432K_1^3-384K_1^2K_0^2+80K_1K_0^4+\frac{16}{3}K_0^6-
6\pi_4K_1K_0+\frac{2\pi_4}{3}K_0^3+\frac{\pi_4^2}{12}\right)K_0.\nonumber
\eea
By comparison (\ref{y60}) and (\ref{y69}), formulas for $K_4$ and $K_6$ in 
(\ref{y69}) may be obtained if we substitute $K_3$ and $K_5$ in (\ref{y69}) into
$K_4$ and $K_6$ in (\ref{y60}). Combining (\ref{y69}) with formulas (\ref{y23}),
we obtain for genera $G_r$ the polynomial expressions in $G_0,G_1$, e.g.,
\bea
G_2&=&\frac{2G_0-1}{3}\left(3G_1-G_0^2+G_0\right),\label{y70}\\
G_3&=&\frac{12G_1^2}{5}+\frac{G_1}{20}\left(10+8G_0^2+2\sigma_1-\sigma_1^2-
4G_0(2+\sigma_1)-\sigma_2\right)-\frac{8G_0^4}{15}+\frac{2G_0^3}{15}+\nonumber\\
&&\frac{G_0^2}{60}\left(\sigma_2+3\sigma_1^2-12\sigma_1-46\right)-\frac{G_0}
{120}\left(\sigma_1(\sigma_2-\sigma_1^2+6\sigma_1-10)+2\pi_4+2\sigma_2-28
\right)+\nonumber\\
&&\frac1{240}\left(\sigma_1(\sigma_2-\sigma_1^2+2\sigma_1-2)+2\pi_4\right).
\nonumber
\eea
\end{example}
\section{Concluding remarks}\label{l6}
We study polynomial identities of arbitrary degree $n$ for syzygies degrees of 
numerical semigroups $S_m$ and show in (\ref{y22}) that for $n\ge m$ they 
contain higher genera $G_r\!=\!\sum_{s\in{\mathbb Z}_>\!\!\setminus S_m}s^r$ of 
$S_m$,
\bea
\sum_{j=1}^{\beta_1}C_{1,j}^n-\sum_{j=1}^{\beta_2}C_{2,j}^n+\ldots-(-1)^{m-1}
\sum_{j=1}^{\beta_{m-1}}C_{m-1,j}^n=\frac{(-1)^mn\:!}{(n-m)\:!}\;\pi_mK_{n-m}(
G_0,\ldots,G_p),\nonumber
\eea
where a coefficient $K_p(G_0,\ldots,G_p)$ is a linear combination of genera. We 
calculate explicitly several first expressions (\ref{y23}) for $K_p$, $0\le p
\le 6$, and put forward Conjecture \ref{con1} related $K_p$ and $G_0,\ldots,
G_p$ for any integer $p\ge 0$. In symbolic calculus \cite{rm78}, this 
relationship (\ref{y27}) reads
\bea
K_p=\left(T(\sigma)+G\right)^p+\frac{2^{p+1}}{p+1}T_{p+1}(\delta),\nonumber
\eea
where after binomial expansion the symbolic powers $T^{p-r}(\sigma)G^r$ are 
converted into $T_{p-r}(\sigma)G_r$. Symmetric polynomials $T_r(\sigma)=T_r(
\sigma_1,\ldots,\sigma_r)$ and $T_r(\delta)=T_r(\delta_1,\ldots,\delta_r)$ are 
arisen in the theory of {\em symmetric CI} semigroups \cite{fl17,fl20}
\bea
\sum_{j=1}^{m-1}e_j^n-\sum_{j>r=1}^{m-1}(e_j+e_r)^n+\ldots-(-1)^{m-1}\left(
\sum_{j=1}^{m-1}e_j\right)^n=\frac{(-1)^m\;n!}{(n-m+1)!}\varepsilon_{m-1}
T_{n-m+1}(E),\nonumber
\eea
where polynomials $T_r(E)\!=\!T_r(E_1,\ldots,E_r)$ are related (see \cite{fl20},
formula (19)) to the polynomial part of the partition function, that gives a 
number of partitions of $s\ge 0$ into $m-1$ positive integers. Thus, in the 
relationship (\ref{y23},\ref{y27}), there coexist genera $G_r$ of {\em 
non-symmetric} semigroup $S_m$ and characteristic polynomials $T_r(\sigma)$ 
associated with semigroup generators $d_1,\ldots,d_m$.

Based on a finite number of syzygies degrees and homogeneity of the $m-2$ first 
polynomial identities (\ref{y6}), we find a number $g_m$ of algebraically
independent coefficients $K_p$ for different kinds of semigroups. Due to the 
relationship (\ref{y27}), this leads to $g_m$ algebraically independent genera 
$G_p$. However, the polynomial equations, related $K_p$, $p\ge g_m$, with 
independent coefficients $K_j$, $0\le j<g_m$, read much shorter than their 
countpartners, related $G_p$, $p\ge g_m$, with independent genera $G_j$. It can 
be seen for non-symmetric and symmetric (not CI) semigroups, (see section 
\ref{l52}), but, in particular, for symmetric CI semigroups comparing 
relations (\ref{y65}), (\ref{y67}) and (\ref{y69}) with (\ref{y66}), 
(\ref{y68}) and 
(\ref{y70}). These observations make us to suppose that $K_p$ has 
deeper algebraic meaning than a simple combination of $G_p$.
\appendix
\renewcommand{\theequation}{\thesection\arabic{equation}}
\section{Appendix: Stirling numbers of the 1st kind}\label{appendix1}
\setcounter{equation}{0}
Making use of recurrence equation (\ref{y12}), we calculate the first formulas 
for $\left[n\atop n-k\right]$ up to $k=9$.
\bea
&&\left[n\atop n\right]=1,\;\;\left[n\atop n-1\right]={n\choose 2},\;\;
\left[n\atop n-2\right]={n\choose 3}\left(\frac{3n}{4}-\frac1{4}\right),\;\;
\left[n\atop n-3\right]={n\choose 4}{n\choose 2},\nonumber\\
&&\left[n\atop n-4\right]={n\choose 5}\left(\frac{5n^3}{16}-\frac{5n^2}{8}+
\frac{5n}{48}+\frac1{24}\right),\nonumber\\
&&\left[n\atop n-5\right]={n\choose 6}\left(\frac{3n^3}{16}-\frac{5n^2}{8}+
\frac{5n}{16}+\frac1{8}\right)n,\label{a1}\\
&&\left[n\atop n-6\right]={n\choose 7}\left(\frac{7n^5}{64}-\frac{35n^4}{64}+
\frac{35n^3}{64}+\frac{91n^2}{576}-\frac{7n}{96}-\frac1{36}\right),\nonumber\\
&&\left[n\atop n-7\right]={n\choose 8}\left(\frac{n^5}{16}-\frac{7n^4}{16}+
\frac{35n^3}{48}+\frac{7n^2}{144}-\frac{7n}{24}-\frac1{9}\right)n.\nonumber\\
&&\left[n\atop n-8\right]={n\choose 9}\left(\frac{9n^7}{256}-\frac{21n^6}{64}+
\frac{105n^5}{128}-\frac{7n^4}{32}-\frac{469n^3}{768}-\frac{9n^2}{64}+
\frac{101n}{960}+\frac{3}{80}\right)\nonumber\\
&&\left[n\atop n-9\right]={n\choose 10}\left(\frac{5n^7}{256}-\frac{15n^6}{64}+
\frac{105n^5}{128}-\frac{7n^4}{12}-\frac{665n^3}{768}+\frac{25n^2}{192}+
\frac{101n}{192}+\frac{3}{16}\right)n.\nonumber
\eea
\vspace{-.5cm}
\renewcommand{\theequation}{\thesection\arabic{equation}}
\section{Appendix: Derivatives $\Psi^{(r)}_{z=1}$}\label{appendix2}
\setcounter{equation}{0}
Find expressions for ratio of derivatives $\Psi^{(r)}_{z=1}/\Psi^{(0)}_{z=1}$, 
$r\le 4$,
\bea
\frac{\Psi^{(1)}_{z=1}}{\Psi^{(0)}_{z=1}}&=&\sum_{j=1}^m\frac{\Psi^{(1)}_{j,z=1}
}{\Psi_{j,z=1}},\hspace{1cm}\frac{\Psi^{(2)}_{z=1}}{\Psi^{(0)}_{z=1}}\;=\;\left(
\sum_{j=1}^m\frac{\Psi^{(1)}_{j,z=1}}{\Psi_{j,z=1}}\right)^2+\sum_{j=1}^m\frac{
\Psi^{(2)}_{j,z=1}}{\Psi_{j,z=1}}-\sum_{j=1}^m\left(\frac{\Psi^{(1)}_{j,z=1}}
{\Psi_{j,z=1}}\right)^2,\nonumber\\
\frac{\Psi^{(3)}_{z=1}}{\Psi^{(0)}_{z=1}}&=&\left(\sum_{j=1}^m\frac{\Psi^{(1)}_
{j,z=1}}{\Psi_{j,z=1}}\right)^3+\sum_{j=1}^m\frac{\Psi^{(3)}_{j,z=1}}{\Psi_{
j,z=1}}+2\sum_{j=1}^m\left(\frac{\Psi^{(1)}_{j,z=1})}{\Psi_{j,z=1}}\right)^3+
3\sum_{j=1}^m\frac{\Psi^{(1)}_{j,z=1}}{\Psi_{j,z=1}}\sum_{j=1}^m\frac{\Psi^{(2)}
_{j,z=1}}{\Psi_{j,z=1}}\nonumber\\
&-&3\sum_{j=1}^m\frac{\Psi^{(1)}_{j,z=1}}{\Psi_{j,z=1}}\sum_{j=1}^m\left(\frac{
\Psi^{(1)}_{j,z=1}}{\Psi_{j,z=1}}\right)^2-3\sum_{j=1}^m\frac{\Psi^{(1)}_{j,z=1}
\Psi^{(2)}_{j,z=1}}{\left(\Psi_{j,z=1}\right)^2},\label{b1}\\
\frac{\Psi^{(4)}_{z=1}}{\Psi^{(0)}_{z=1}}&=&\left(\sum_{j=1}^m\frac{\Psi^{(1)}_
{j,z=1}}{\Psi_{j,z=1}}\right)^4+\sum_{j=1}^m\frac{\Psi^{(4)}_{j,z=1}}{\Psi_{j,z
=1}}-6\sum_{j=1}^m\left(\frac{\Psi^{(1)}_{j,z=1}}{\Psi_{j,z=1}}\right)^4+
3\left(\sum_{j=1}^m\left(\frac{\Psi^{(1)}_{j,z=1}}{\Psi_{j,z=1}}\right)^2\right)
^2\nonumber\\
&-&6\sum_{j=1}^m\left(\frac{\Psi^{(1)}_{j,z=1}}{\Psi_{j,z=1}}\right)^2\left(
\sum_{j=1}^m\frac{\Psi^{(1)}_{j,z=1}}{\Psi_{j,z=1}}\right)^2+8\sum_{j=1}^m
\frac{\Psi^{(1)}_{j,z=1}}{\Psi_{j,z=1}}\sum_{j=1}^m\left(\frac{\Psi^{(1)}_{j,z=
1}}{\Psi_{j,z=1}}\right)^3\nonumber\\
&+&6\left(\sum_{j=1}^m\frac{\Psi^{(1)}_{j,z=1}}{\Psi_{j,z=1}}\right)^2\sum_{j=1}
^m\frac{\Psi^{(2)}_{j,z=1}}{\Psi_{j,z=1}}+12\sum_{j=1}^m\left(\frac{\Psi^{(1)}
_{j,z=1}}{\Psi_{j,z=1}}\right)^2\frac{\Psi^{(2)}_{j,z=1}}{\Psi_{j,z=1}}+3
\left(\sum_{j=1}^m\frac{\Psi^{(2)}_{j,z=1}}{\Psi_{j,z=1}}\right)^2\nonumber\\
&-&6\sum_{j=1}^m\left(\frac{\Psi^{(1)}_{j,z=1}}{\Psi_{j,z=1}}\right)^2
\sum_{j=1}^m\frac{\Psi^{(2)}_{j,z=1}}{\Psi_{j,z=1}}-3\sum_{j=1}^m\left(\frac{
\Psi^{(2)}_{j,z=1}}{\Psi_{j,z=1}}\right)^2-4\sum_{j=1}^m\frac{\Psi^{(1)}_{j,z=1}
\Psi^{(3)}_{j,z=1}}{\left(\Psi_{j,z=1}\right)^2}\nonumber\\
&-&12\sum_{j=1}^m\frac{\Psi^{(1)}_{j,z=1}}{\Psi_{j,z=1}}\sum_{j=1}^m\frac{
\Psi^{(1)}_{j,z=1}\Psi^{(2)}_{j,z=1}}{\left(\Psi_{j,z=1}\right)^2}+4\sum_{j=1}^m
\frac{\Psi^{(1)}_{j,z=1}}{\Psi_{j,z=1}}\sum_{j=1}^m\frac{\Psi^{(3)}_{j,z=1}}
{\Psi_{j,z=1}}.\nonumber
\eea
Using a summation rule in a finite calculus (see formula (2.50) in \cite{gr94}),
we find a ratio $\Psi^{(r)}_{j,z=1}/\Psi_{j,z=1}$,
\bea
\sum_{l=0}^{d-1}(l)_r=\frac{(d)_{r+1}}{r+1}\qquad\longrightarrow\qquad
\frac{\Psi^{(r)}_{j,z=1}}{\Psi_{j,z=1}}=\frac{(d_j-1)_r}{r+1}.\label{b2}
\eea
Substituting (\ref{b2}) into (\ref{b1}) we arrive at formulas (\ref{y17}).
\vspace{-.2cm}
\renewcommand{\theequation}{\thesection\arabic{equation}}
\section{Appendix: Symmetric polynomials $T_k(X_1,\ldots,X_m)$}\label{appendix3}
\setcounter{equation}{0}
We present formulas for the first symmetric polynomials $T_k(X_1,\ldots,X_m)$ 
up to $k=7$.
\bea
&&T_0=1,\label{c1}\\
&&T_1=\frac1{2}X_1,\nonumber\\
&&T_2=\frac1{3}\frac{3X_1^2+X_2}{4},\nonumber\\
&&T_3=\frac1{4}\frac{X_1^2+X_2}{2}\;X_1,\nonumber\\
&&T_4=\frac1{5}\frac{15X_1^4+30X_1^2X_2+5X_2^2-2X_4}{48},\nonumber\\
&&T_5=\frac1{6}\frac{3X_1^4+10X_1^2X_2+5X_2^2-2X_4}{16}\;X_1,\nonumber\\
&&T_6=\frac1{7}\frac{63X_1^6+315X_1^4X_2+315X_1^2X_2^2-126X_1^2X_4+35X_2^3-
42X_2X_4+16X_6}{576},\nonumber\\
&&T_7=\frac1{8}\frac{9X_1^6+63X_1^4X_2+105X_1^2X_2^2-42X_1^2X_4+35X_2^3-42X_2X_4
+16X_6}{144}\;X_1.\nonumber
\eea
\vspace{-1cm}


\begin{thebibliography}{99}
\bibitem{be752} H. Bresinsky, {\it Symmetric semigroups of integers generated
               by four elements}, Manuscripta Math., {\bf 17}, 205-219 (1975)
\bibitem{gkz4} I.M. Gelfand, M.M. Kapranov, A.V. Zelevinski, {\it Discriminants,
               Resultants and Multidimensional Determinants}, BirkhWauser, 
               Boston, 1994.
\bibitem{gr94} R.L. Graham, D.E. Knuth \& O. Patashnik,{\it Concrete 
             Mathematics: a foundation for computer science}, Addison-Wesley, 
               NY, 2nd ed., 1994
\bibitem{fr07} L.G. Fel and B.Y. Rubinstein, {\em Power sums related to
          semigroups $S(d_1,d_2,d_3)$}, Semigroup Forum, {\bf 74}, 93-98 (2007)
\bibitem{fl10} L.G. Fel, {\it Duality relation for the Hilbert series of almost 
               symmetric numerical semigroups}, Israel J. Math. {\bf 185}, 
               413-444 (2011).
\bibitem{fl15} L.G. Fel, {\it On Frobenius numbers for symmetric (not complete 
               intersection) semigroups generated by four elements} Semigroup 
               Forum, {\bf 93}, 423-426 (2016).
\bibitem{fl17} L.G. Fel, {\it Restricted partition functions and identities
               for degrees of syzygies in numerical semigroups}, Ramanujan
               J., {\bf 43}, 465-491 (2017)
\bibitem{f18}  L.G. Fel, {\it Symmetric (not complete intersection) semigroups 
               generated by five elements}, Integers: The Electronic J. of 
               Comb. Number Theory, 18 (2018), \# A44.
\bibitem{fa18} L.G. Fel, {\it Symmetric (not complete intersection) semigroups 
               generated by six elements}, In {\it Numerical Semigroups}, 
               Springer INdAM Series {\bf 40}, 93-109 (2020)
\bibitem{fl18} L.G. Fel, {\it Symmetric (not complete intersection) semigroups
               generated by 4,5,6 elements}, International Meeting Numerical 
               Semigroups, 2018, Cortona, Italy, 
               https://www.ugr.es/~imns2010/2018
\bibitem{fl20} L.G. Fel, {\it Symmetric polynomials associated with numerical 
              semigroups}, preprint, 2020\\ https://arxiv.org/pdf/2010.03363.pdf
\bibitem{hej18}B. Hejmej, {\it A note about irreducibility of a resultant}, 
              Bulletin dela Soci\'et\'e des Sciences et des Letters de L\'od\'z,
               {\bf 68}, 27-32 (2018)
\bibitem{Her84} J. Herzog and M. K\"uhl, {\em On the Betti numbers of finite 
               pure and linear resolutions}, Comm. Algebra, {\bf 12},
               1627-1646, (1984)
\bibitem{mc95} I.G. Macdonald, {\em Symmetric functions and Hall polynomials},
               Oxford: Clarendon Press, 1995
\bibitem{ma02} F. S. Macaulay, {\em Some formul\ae\hspace{.1cm}in elimination}, 
               Proc. London Math. Soc., {\bf 35}, 3-27 (1902)
\bibitem{ro94} \"{O}. J. R\"{o}dseth, {\em A note on Brown and Shiue's paper on 
               a remark related to the Frobenius problem}, Fibonacci Quarterly, 
               {\bf 32}, 407-408 (1994)
\bibitem{rm78} S.M. Roman and G.-C. Rota, {\em The umbral calculus}, Adv. Math. 
               {\bf 27}, 95-188 (1978).
\bibitem{rs09} J. C. Rosales and P.A. Garc\'ia-S\'anchez,  {\em Numerical  
               semigroups}, Springer, New York, 2009.
\end{thebibliography}
\end{document}